\documentclass[a4paper,10pt]{article}

\usepackage{authblk}

\usepackage[utf8]{inputenc} 
\usepackage[T1]{fontenc}    
\usepackage{indentfirst}
\usepackage{lmodern}

\usepackage{mathtools}
\usepackage{verbatim}
\usepackage{url}            
\usepackage{booktabs}       
\usepackage{amsfonts}       
\usepackage{amsmath}
\usepackage{amssymb}
\usepackage{amsthm}
\usepackage{microtype}      
\usepackage[a4paper]{geometry} 
\usepackage{enumitem}
\usepackage{xcolor}
\usepackage{mathrsfs}

\usepackage{geometry}
\usepackage{fancyhdr}
\usepackage{listings}

\usepackage[linktocpage=true]{hyperref}
\hypersetup{
    colorlinks=true,
    linkcolor=blue,
    citecolor=violet,
    urlcolor=blue
}

\usepackage{graphicx}
\usepackage[makeroom]{cancel}
\usepackage{tikz}

\usepackage[inkscapeformat=png]{svg}
\usepackage[normalem]{ulem}

\theoremstyle{plain}
\newtheorem{theorem}{Theorem}[section]
\newtheorem{proposition}[theorem]{Proposition}
\newtheorem{corollary}[theorem]{Corollary}

\theoremstyle{definition}

\theoremstyle{remark}
\newtheorem{remark}[theorem]{Remark}

\makeatletter
\def\keywords{\xdef\@thefnmark{}\@footnotetext}
\makeatother

\title{Two Hornich-Hlawka-type and Gram matrix-based inequalities}
\author{Nizar El Idrissi and Hicham Zoubeir}

\newcommand{\Addresses}{{
  \bigskip
  \footnotesize
  \textbf{Nizar El Idrissi.}
  \par\nopagebreak 
  Laboratoire : Equations aux dérivées partielles, Algèbre et Géométrie spectrales.
  \par\nopagebreak 
  Département de mathématiques, faculté des sciences, université Ibn Tofail, Kénitra, Maroc.
  \par\nopagebreak 
  \textit{Adresse e-mail} : \texttt{nizar.elidrissi@uit.ac.ma}

  \medskip
  \footnotesize
  \textbf{Hicham Zoubeir.}
  \par\nopagebreak 
  Laboratoire : Equations aux dérivées partielles, Algèbre et Géométrie spectrales.
  \par\nopagebreak 
  Département de mathématiques, faculté des sciences, université Ibn Tofail, Kénitra, Maroc.
  \par\nopagebreak 
  \textit{Adresse e-mail} : \texttt{hzoubeir2014@gmail.com}
}}

\begin{document}

\newpage
\maketitle

\begin{abstract}
We establish two inequalities in real inner product spaces. The first is a multiplicative strengthening of the classical Hornich–Hlawka inequality: for all vectors $x,y,z$ in a real inner product space $H$
\[ \|x\|\,\|y\| + \|z\|\,\|x+y+z\| \;\geq\; \|x+z\|\,\|y+z\|. \]
We provide a complete characterization of the equality cases. The second result is a parametric inequality derived from the positive semidefiniteness of Gram matrices: for all $x,y,z \in H$ and $\alpha, \beta, \gamma \in \mathbb{R}$, 
\[ \alpha^2\|x\|^2\langle y,z\rangle^2 + \beta^2\|y\|^2\langle x,z\rangle^2 + \gamma^2\|z\|^2\langle x,y\rangle^2 + 2(\alpha\beta + \alpha\gamma + \beta\gamma)\langle x,y\rangle\langle x,z\rangle\langle y,z\rangle \;\geq\; 0, \]
which, upon optimization over the parameters yields sharp reverse inequalities to the Gram determinant inequality $\det G \geq 0$. As a special case, this recovers and strengthens the classical Cauchy–Schwarz inequality.
\end{abstract}


\tableofcontents

\section{Introduction}

The Hornich–Hlawka inequality — sometimes called Hlawka's inequality or the quadrilateral inequality — is a classical result in the geometry of normed spaces. In its standard form, it asserts that in any normed vector space $E$ in which the inequality holds,
\[\|x\| + \|y\| + \|z\| + \|x+y+z\| \;\geq\; \|x+y\| + \|x+z\| + \|y+z\| \qquad \forall\, x,y,z \in E.\]
The inequality was originally established by Hornich \cite{Hornich1942} who referred to Hlawka as the originator of the argument, and has since generated a substantial literature, being studied extensively from the perspectives of functional analysis, convex geometry, and the theory of inequalities. Fechner \cite{Fechner2014} initiated the systematic study of the associated functional inequality $f(x+y)+f(y+z)+f(x+z) \leq f(x+y+z)+f(x)+f(y)+f(z)$, investigating it for unknown mappings $f$ defined on Abelian groups and linear spaces. In \cite{Fechner2017}, he further analyzed the regularity of solutions of this functional inequality on topological abelian groups, establishing conditions under which measurability or local boundedness forces continuity, and characterizing the lower and upper hull of $f$. Fechner and Jabłońska \cite{FechnerJablonska2025} extended this line of inquiry to the almost-everywhere setting, studying the functional inequality for mappings $f:\mathbb{R}\to\mathbb{R}$ that need only be satisfied outside a Lebesgue null set. In a different direction, Honda, Okazaki and Takahashi \cite{HondaOkazakiTakahashi1998} generalized Hlawka's three-vector inequality in Hilbert space to an $n$-element inequality in $L^1$, relating it to Hanner's inequality and to the Adamović inequality. Ressel \cite{Ressel2015} showed that the Hornich–Hlawka inequality, initially stated for the identity function on the real line, extends to the entire class of Bernstein functions on the half-line, and proved in particular that it holds for the square-root function when applied to vectors in Euclidean space. Takahashi, Takahasi and Wada \cite{TakahashiTakahasiWada2009} investigated a general integral form of Hlawka's inequality on Banach spaces, establishing both the inequality and its reverse and clarifying the connection with the Djoković inequality in a weighted integral setting. Serre \cite{Serre2015} considered a genuinely different geometric setting: he proved that in the future cone of Minkowski space, the Lorentzian pseudo-norm satisfies a "reverse" Hlawka-type inequality — the orientation of the inequality is opposite to the Euclidean case, in direct analogy with the reverse of the Cauchy–Schwarz inequality that holds in that setting. While the inequality fails in certain normed spaces, it is known to hold in all inner product spaces, as well as in broader classes including $L^p$ spaces for $p \in [1,2]$. \\ \\
The other inequality discussed in the present paper belongs to the tradition of inequalities in inner product spaces derived from structural properties of Gram matrices. The nonnegativity of the Gram determinant $\det G \geq 0$ and the positive semidefiniteness of Gram matrices are classical tools for generating vector inequalities; see, e.g., \cite{AravHallLi2008}, where a Cauchy-Schwarz-type inequality for triples of vectors is established and interpreted geometrically in terms of the vertex angles of a tetrahedron, with Gram matrices playing a central role. A major impetus for the systematic study of such inequalities is provided by Buzano's inequality and its relatives. Dragomir \cite{Dragomir2005} established refinements of both Buzano's inequality and Kurepa's inequality in inner product spaces, using Gramian-based arguments to obtain improvements of Grüss-type inequalities and reverse Schwarz inequalities. Sababheh, Moradi and Heydarbeygi \cite{SababhehMoradiHeydarbeygi2022} refined the Cauchy-Schwarz and Buzano inequalities simultaneously and provided a new proof of a refined Kreĭn-type inequality, with applications to operator norm and numerical radius inequalities for Hilbert space operators. For a comprehensive treatment of the broader landscape of Schwarz-type inequalities in inner product spaces, including reverses, refinements, and connections to Grüss and Bessel inequalities, we refer to the monograph \cite{Dragomir2005book}. The inequality established in Section~\ref{section-3} of the present paper can be situated within this framework: it derives from the positive semidefiniteness of the Gram matrix via a single substitution, and its corollaries constitute sharp reverse inequalities to the Gram determinant condition, recovering the Cauchy-Schwarz inequality as a special case. \\ \\
The present paper makes two independent contributions to the theory of inequalities in inner product spaces. \\ \\
In Section~\ref{section-2}, we prove that in any real inner product space $H$, the Hornich–Hlawka inequality admits the following multiplicative strengthening: for all $x,y,z \in H$,
\begin{equation}
\label{inequality-strong-hlawka}
\|x\|\,\|y\| + \|z\|\,\|x+y+z\| \;\geq\; \|x+z\|\,\|y+z\|.
\end{equation}
Inequality \ref{inequality-strong-hlawka} is stronger than the classical Hornich-Hlawka inequality. Indeed, we show that it implies the Hornich–Hlawka inequality via a direct algebraic argument. We also provide a complete description of the equality cases (Proposition~\ref{proposition-equality-case}), showing that equality holds if and only if $x,y,z$ are linearly dependent and satisfy explicit algebraic conditions on their norms and inner products. In particular, equality can only occur in flat configurations, i.e., when the three vectors lie in a subspace of dimension at most two. \\ \\
In Section~\ref{section-3}, we prove a parametric family of inequalities that follows directly from the positive semidefiniteness of the Gram matrix $G$ associated to three vectors $x,y,z \in H$. Namely, we show that for all $\alpha, \beta, \gamma \in \mathbb{R}$,
\[\alpha^2\|x\|^2\langle y,z\rangle^2 + \beta^2\|y\|^2\langle x,z\rangle^2 + \gamma^2\|z\|^2\langle x,y\rangle^2 + 2(\alpha\beta + \alpha\gamma + \beta\gamma)\langle x,y\rangle\langle x,z\rangle\langle y,z\rangle \;\geq\; 0.\]
When $\langle x,y\rangle\langle x,z\rangle\langle y,z\rangle > 0$, the choice $(\alpha,\beta,\gamma) = (1,1,-2)$ (and its cyclic permutations) yields the sharp inequality 
\begin{equation}
\|x\|^2\langle y,z\rangle^2 + \|y\|^2\langle x,z\rangle^2 + \|z\|^2\langle x,y\rangle^2 \;\geq\; 3\langle x,y\rangle\langle x,z\rangle\langle y,z\rangle, 
\end{equation} 
while when $\langle x,y\rangle\langle x,z\rangle\langle y,z\rangle < 0$, the choice $(\alpha,\beta,\gamma) = (1,1,1)$ yields the sharp inequality
\begin{equation}
\|x\|^2\langle y,z\rangle^2 + \|y\|^2\langle x,z\rangle^2 + \|z\|^2\langle x,y\rangle^2 \;\geq\; -6\langle x,y\rangle\langle x,z\rangle\langle y,z\rangle.
\end{equation}
Both inequalities are sharp, as demonstrated by explicit examples. These results can be understood as reverse inequalities to the Gram determinant condition $\det G \geq 0$, and the first of them implies the Cauchy–Schwarz inequality as the special case $z = y$. \\
\mbox{} \\
Throughout this paper, $H$ denotes a real inner product space, with inner product $\langle \cdot, \cdot \rangle$ and associated norm $\|\cdot\| = \sqrt{\langle \cdot, \cdot \rangle}$.

\section{A stronger Hornich-Hlawka inequality}
\label{section-2}

The following inequality is a multiplicative strengthening of the classical Hornich-Hlawka inequality.

\begin{theorem}
\label{theorem-strong-hlawka}
For all $x,y,z \in H$, we have
\[ \lVert x \rVert \lVert y \rVert + \lVert z \rVert \lVert x+y+z \rVert \geq \lVert x + z \rVert \lVert y + z \rVert, \]
\end{theorem}

\begin{proof}
We first reduce \ref{inequality-strong-hlawka} to a polynomial inequality in the entries of the Gram matrix of $(x,y,z)$. \\
Denote by $L$ and $R$ the left and right hand sides of this inequality, respectively. \\
We have
{\tiny
\begin{align*}
L^2 &= \lVert x \rVert ^2 \lVert y \rVert^2 + \lVert z \rVert^2 \left( \lVert x \rVert^2 + \lVert y \rVert^2 + \lVert z \rVert^2 + 2 \langle x,y \rangle + 2 \langle x,z \rangle + 2\langle y,z \rangle  \right) + \\
&\quad \quad 2 \lVert x \rVert \lVert y \rVert \lVert z \rVert \lVert x + y + z \rVert \\
&=  \lVert x \rVert ^2 \lVert y \rVert^2 + \lVert z \rVert^2 \lVert x \rVert^2 + \lVert z \rVert^2 \lVert y \rVert^2 + \lVert z \rVert^4 + \\
&\quad \quad 2\lVert z \rVert^2 \langle x , y \rangle +  2\lVert z \rVert^2 \langle x , z \rangle +  2\lVert z \rVert^2 \langle y , z \rangle + 2 \lVert x \rVert \lVert y \rVert \lVert z \rVert \lVert x + y + z \rVert
\end{align*}
}
and
{\scriptsize
\begin{align*}
R^2 &= \left( \lVert x \rVert^2 + \lVert z \rVert^2 + 2\langle x , z \rangle \right) \left( \lVert y \rVert^2 + \lVert z \rVert^2 + 2\langle y , z \rangle \right) \\
&= \lVert x \rVert ^2 \lVert y \rVert^2 + \lVert z \rVert^2 \lVert x \rVert^2 + \lVert z \rVert^2 \lVert y \rVert^2 + \lVert z \rVert^4  + 2 \langle x , z \rangle \lVert y \rVert^2 + 2\langle x , z \rangle \lVert z \rVert^2 \\
&\quad \quad + 2 \langle y , z \rangle \lVert x \rVert^2 + 2 \langle y , z \rangle \lVert z \rVert^2 + 4 \langle x , z \rangle \langle y , z \rangle
\end{align*}
}
We have
{\scriptsize
\begin{align*}
&L \geq R \\
&\Leftrightarrow L^2 \geq R^2 \\
&\Leftrightarrow \frac{L^2 - \left( \lVert x \rVert ^2 \lVert y \rVert^2 + \lVert z \rVert^2 \lVert x \rVert^2 + \lVert z \rVert^2 \lVert y \rVert^2 + \lVert z \rVert^4 \right)}{2} \geq \\
&\quad \quad \quad \frac{R^2 - \left( \lVert x \rVert ^2 \lVert y \rVert^2 + \lVert z \rVert^2 \lVert x \rVert^2 + \lVert z \rVert^2 \lVert y \rVert^2 + \lVert z \rVert^4 \right)}{2} \\
&\Leftrightarrow \lVert z \rVert^2 \langle x , y \rangle +  \lVert z \rVert^2 \langle x , z \rangle +  \lVert z \rVert^2 \langle y , z \rangle + \lVert x \rVert \lVert y \rVert \lVert z \rVert \lVert x + y + z \rVert \geq \\
&\quad \quad \quad \langle x , z \rangle \lVert y \rVert^2 + \langle x , z \rangle \lVert z \rVert^2 +  \langle y , z \rangle \lVert x \rVert^2 + \langle y , z \rangle \lVert z \rVert^2 + 2 \langle x , z \rangle \langle y , z \rangle \\
&\Leftrightarrow \lVert z \rVert^2 \langle x , y \rangle + \lVert x \rVert \lVert y \rVert \lVert z \rVert \lVert x + y + z \rVert \geq \lVert x \rVert^2 \langle y , z \rangle + \lVert y \rVert^2 \langle x , z \rangle + 2 \langle x , z \rangle \langle y , z \rangle
\end{align*}
}
This inequality is equivalent to
{\scriptsize
\[ \Leftrightarrow \lVert x \rVert \lVert y \rVert \lVert z \rVert \lVert x + y + z \rVert \geq \lVert x \rVert^2 \langle y , z \rangle + \lVert y \rVert^2 \langle x , z \rangle - \lVert z \rVert^2 \langle x , y \rangle + 2 \langle x , z \rangle \langle y , z \rangle \]
}
Denote by $L_1$ and $R_1$ the left and right hand sides of this last inequality, respectively. \\
If $R_1 \leq 0$, the inequality is trivial, so we can suppose that $R_1 \geq 0$. \\
Set 
\[ a = \lVert x \rVert, \quad b = \lVert y \rVert, \quad c = \lVert z \rVert, \quad p = \langle x,y \rangle, \quad q = \langle x,z \rangle, \quad r = \langle y,z \rangle. \]
We have 
\begin{align*}
L_1^2 &=  \lVert x \rVert^2  \lVert y \rVert^2  \lVert z \rVert^2  \left( \lVert x \rVert^2 + \lVert y \rVert^2  + \lVert z \rVert^2 + 2  \langle x , y \rangle + 2 \langle x , z \rangle +  2 \langle y , z \rangle \right) \\
&= a^2 b^2 c^2 (a^2 + b^2 + c^2 + 2p + 2q + 2r).
\end{align*}
and
{\scriptsize
\begin{align*}
R_1^2 &= \lVert x \rVert^4 \langle y , z \rangle^2 +  \lVert y \rVert^4 \langle x , z \rangle^2 +  \lVert z \rVert^4 \langle x , y \rangle^2 +  4 \langle x , z \rangle^2 \langle y , z \rangle^2 \\
&\quad + 2 \lVert x \rVert^2 \lVert y \rVert^2 \langle x , z \rangle \langle y , z \rangle - 2 \lVert x \rVert^2 \lVert z \rVert^2  \langle x , y \rangle  \langle y , z \rangle   - 2 \lVert y \rVert^2 \lVert z \rVert^2 \langle x , y \rangle  \langle x , z \rangle   \\
&\quad + 4 \lVert x \rVert^2  \langle x , z \rangle \langle y , z \rangle ^2 + 4 \lVert y \rVert^2 \langle y , z \rangle \langle x , z \rangle ^2 - 4 \lVert z \rVert^2 \langle x , y \rangle \langle x , z \rangle  \langle y , z \rangle  \\
&=  a^4 r^2 + b^4 q^2 + c^4 p^2 + 4q^2 r^2 + 2a^2 b^2 qr - 2a^2 c^2 pr - 2b^2 c^2 pq + \\
&\quad \quad 4a^2 q r^2 + 4b^2 r q^2 - 4c^2 pqr
\end{align*}
}
Consider 
\begin{align*}
\xi &= L_1^2 - R_1^2 \\
&= a^2 b^2 c^2 (a^2 + b^2 + c^2 + 2p + 2q + 2r) \\
&\quad - \left( a^4 r^2 + b^4 q^2 + c^4 p^2 + 4q^2 r^2 + 2a^2 b^2 qr - 2a^2 c^2 pr - 2b^2 c^2 pq \right) \\
&\quad \quad - \left( 4a^2 q r^2 + 4b^2 r q^2 - 4c^2 pqr \right)
\end{align*}

It remains to prove that
\[
\xi\ge0\ \text{ if }\det G=0,\qquad \xi>0\ \text{ if }\det G\neq0 .
\]

Let $w = x+y+z$.

\emph{Case $\det G=0$.} Then $x,y,z$ are linearly dependent, so one of them is a linear
combination of the other two. Since $L_1$ and $R_1$ are invariant under the exchange
$x\leftrightarrow y$, it suffices to treat the cases $z=\lambda x+\mu y$ and
$x=\lambda y+\mu z$. We use the identity
$\|u\|^2\|v\|^2-\langle u,v\rangle^2=(\lambda\mu'-\lambda'\mu)^2\big(\|s\|^2\|t\|^2-\langle s,t\rangle^2\big)$
for $u=\lambda s+\mu t$, $v=\lambda's+\mu't$.

\smallskip
\noindent\emph{(A) $z=\lambda x+\mu y$.} Then $w=(\lambda+1)x+(\mu+1)y$,
$q=\lambda a^2+\mu p$, $r=\lambda p+\mu b^2$, $c^2=\lambda^2a^2+2\lambda\mu p+\mu^2b^2$, and
\begin{align*}
a^2r+b^2q-c^2p&=(\lambda+\mu)a^2b^2+\lambda(1-\lambda)a^2p+\mu(1-\mu)b^2p-2\lambda\mu p^2,\\
2qr&=2\lambda\mu\,a^2b^2+2\lambda^2a^2p+2\mu^2b^2p+2\lambda\mu p^2 .
\end{align*}
Adding, and putting $m=\lambda(\lambda+1)a^2$, $n=\mu(\mu+1)b^2$, $\sigma=\lambda+\mu+2\lambda\mu$:
\begin{equation}\label{eq:A1}
R_1=\sigma a^2b^2+(m+n)p .
\end{equation}
Next, $\langle z,w\rangle=m+n+\sigma p$, and the identity above gives
$c^2\|w\|^2=(m+n+\sigma p)^2+(\lambda-\mu)^2(a^2b^2-p^2)$. Hence
\begin{align*}
\xi&=a^2b^2\,c^2\|w\|^2-R_1^2\\
&=(a^2b^2-p^2)\big[(m+n)^2-\sigma^2a^2b^2+(\lambda-\mu)^2a^2b^2\big].
\end{align*}
As $\sigma^2-(\lambda-\mu)^2=4\lambda\mu(\lambda+1)(\mu+1)$ and $4mn=4\lambda\mu(\lambda+1)(\mu+1)a^2b^2$, the
bracket equals $(m+n)^2-4mn=(m-n)^2$, so
\begin{equation}\label{eq:A2}
\xi=(a^2b^2-p^2)\,\big(\lambda(\lambda+1)a^2-\mu(\mu+1)b^2\big)^2 .
\end{equation}

\smallskip
\noindent\emph{(B) $x=\lambda y+\mu z$.} Then $w=(\lambda+1)y+(\mu+1)z$,
$p=\lambda b^2+\mu r$, $q=\lambda r+\mu c^2$, $a^2=\lambda^2b^2+2\lambda\mu r+\mu^2c^2$. Put
\[
E=\lambda(\lambda+1)b^2+\mu(\mu+1)c^2+2\lambda(\mu+1)r .
\]
Substituting in $R_1=a^2r+b^2q-c^2p+2qr$, where
$a^2r=\lambda^2b^2r+2\lambda\mu r^2+\mu^2c^2r$, $b^2q=\lambda b^2r+\mu b^2c^2$,
$c^2p=\lambda b^2c^2+\mu c^2r$, $2qr=2\lambda r^2+2\mu c^2r$, we get
\begin{equation}\label{eq:B1}
R_1=(\mu-\lambda)\,b^2c^2+rE .
\end{equation}
Moreover $\langle x,w\rangle=E+(\mu-\lambda)r$, so by the identity above
$a^2\|w\|^2=(E+(\mu-\lambda)r)^2+(\mu-\lambda)^2(b^2c^2-r^2)$. Therefore
\[
\xi=b^2c^2\big[(E+(\mu-\lambda)r)^2+(\mu-\lambda)^2(b^2c^2-r^2)\big]-\big((\mu-\lambda)b^2c^2+rE\big)^2 ,
\]
and expanding, all terms containing $\mu-\lambda$ cancel:
\begin{equation}\label{eq:B2}
\xi=(b^2c^2-r^2)\,E^2 .
\end{equation}
In both cases $\xi\ge0$ by the Cauchy--Schwarz inequality.

\medskip
\emph{Case $\det G\neq0$.} Then $z\neq0$ (otherwise the third row of $G$ vanishes), so $c>0$.
Fix $a,b,c,q,r$ and regard $\xi$ as a function of $p$:
\[
\xi(p)=a^2b^2c^2\,(a^2+b^2+c^2+2p+2q+2r)-(a^2r+b^2q-c^2p+2qr)^2 .
\]
The first term is affine in $p$ and $R_1=-c^2p+\text{const}$, so $\xi$ is a quadratic polynomial with
leading coefficient $-c^4<0$: it is strictly concave on $\mathbb R$. (If $c=0$ it would only be
affine, hence concave, but this case does not occur here.)

Let $G(p)$ be the matrix $G$ with $\langle x,y\rangle$ replaced by $p$, and
$K=\{p\in\mathbb R:\ G(p)\text{ is positive semidefinite}\}$. The map $p\mapsto G(p)$ is affine and the
cone of positive semidefinite matrices is closed and convex, so $K$ is a closed interval. It is bounded
since $p^2\le a^2b^2$, and it contains $p_0=\langle x,y\rangle$. Write $K=[p_-,p_+]$.
The matrices $G(p_\pm)$ are singular: if $G(p_+)$ were positive definite, then, positive definiteness
being an open condition, $G(p_++\varepsilon)$ would be positive definite for small $\varepsilon>0$,
contradicting the maximality of $p_+$; the same holds for $p_-$. Since $G(p_0)=G$ is positive
semidefinite and nonsingular, we get $p_-<p_0<p_+$.

For $p\in\{p_-,p_+\}$, $G(p)$ is the Gram matrix of three vectors $x',y',z'\in\mathbb R^3$ (take the
columns of $G(p)^{1/2}$), which are linearly dependent because $G(p)$ is singular. Since
$\|x'+y'+z'\|^2=a^2+b^2+c^2+2p+2q+2r$, the number $\xi(p)$ is the quantity $\xi$ of $(x',y',z')$, so
$\xi(p_\pm)\ge0$ by the case $\det G=0$. Writing $p_0=(1-t)p_-+tp_+$ with $0<t<1$, strict concavity gives
\[
\xi(p_0)>(1-t)\,\xi(p_-)+t\,\xi(p_+)\ \ge 0 .
\]

In all cases $L_1^2-R_1^2=\xi\ge0$, so $L_1\ge|R_1|\ge R_1$, and therefore
$L^2-R^2=2(L_1-R_1)\ge0$, i.e.\ $L\ge R$.
\end{proof}

\begin{remark} \mbox{}\\*
\begin{enumerate}
\item When $H = \mathbb{R}$ or $\mathbb{C}$, the inequality can be easily proved by expanding both sides and using the multiplicativity of the modulus (this property is senseless when we consider vector norms instead of the modulus of complex numbers).
\item 
The strong Hornich-Hlawka inequality implies the classical Hornich-Hlawka inequality. Indeed, since both sides of the classical Hornich-Hlawka inequality are nonnegative, we may replace them with their squares. \\
For the left-hand side, we obtain
{\tiny
\[
\bigl( \|x\| + \|y\| + \|z\| + \|x + y + z\| \bigr)^2 = 2\bigl(\|x\|^2 + \|y\|^2 + \|z\|^2\bigr) + 2\bigl(\langle x,y\rangle + \langle y,z\rangle + \langle z,x\rangle\bigr)
\]
\[
+ 2\bigl( \|x\| + \|y\| + \|z\| \bigr) \|x + y + z\| + 2\bigl( \lVert x \rVert \lVert y \rVert + \lVert x \rVert \lVert z \rVert + \lVert y \rVert \lVert z \rVert \bigr).
\]
}
For the right-hand side, we have
{\tiny
\[
\bigl( \|x + y\| + \|y + z\| + \|z + x\| \bigr)^2 = 2\bigl(\|x\|^2 + \|y\|^2 + \|z\|^2\bigr) + 2\bigl(\langle x,y\rangle + \langle y,z\rangle + \langle z,x\rangle\bigr)
\]
\[
+ 2\bigl( \|x + y\| \cdot \|y + z\| + \|y + z\| \cdot \|z + x\| + \|z + x\| \cdot \|x + y\| \bigr).
\]
}
Thus, the classical Hornich-Hlawka inequality is equivalent to
\[
\bigl( \|x\| + \|y\| + \|z\| \bigr) \|x + y + z\| + \bigl( \lVert x \rVert \lVert y \rVert +  \lVert x \rVert \lVert z  \rVert + \lVert y \rVert \lVert z \rVert \bigr).
\]
\[
\geq \|x + y\| \cdot \|y + z\| + \|y + z\| \cdot \|z + x\| + \|z + x\| \cdot \|x + y\|
\]
which, after regrouping, is equivalent to
\[
\bigl( \lVert x \rVert \lVert y \rVert  + \|z\| \cdot \|x + y + z\| - \|y + z\| \cdot \|z + x\| \bigr) +
\]
\[
\bigl( \lVert y \rVert \lVert z \rVert + \|x\| \cdot \|x + y + z\| - \|z + x\| \cdot \|x + y\| \bigr) +
\]
\[
\bigl( \lVert x \rVert \lVert z \rVert  + \|y\| \cdot \|x + y + z\| - \|x + y\| \cdot \|y + z\| \bigr) \geq 0.
\]
Therefore, it is clear that the strong Hornich-Hlawka inequality implies the classical Hornich-Hlawka inequality.
\end{enumerate}
\end{remark}

\begin{proposition}\label{proposition-equality-case}
Let $x,y,z\in H$ and $w=x+y+z$. Equality holds in \eqref{inequality-strong-hlawka} if and only if
$x,y,z$ are linearly dependent and one of the following holds.
\begin{enumerate}
\item[(a)] $x,y$ are linearly independent, $z=\lambda x+\mu y$ with
  $\|x\|^2\lambda(\lambda+1)=\|y\|^2\mu(\mu+1)$, and $(\lambda+\mu+2\lambda\mu)\|x\|^2\|y\|^2+\big(\lambda(\lambda+1)\|x\|^2+\mu(\mu+1)\|y\|^2\big)\langle x,y\rangle \ge 0$.
\item[(b)] $y,z$ are linearly independent, $x=\lambda y+\mu z$ with
  \[
  \lambda(\lambda+1)\|y\|^2+\mu(\mu+1)\|z\|^2+2\lambda(\mu+1)\langle y,z\rangle=0
  \quad\text{and}\quad \mu\ge\lambda .
  \]
\item[(c)] $x,z$ are linearly independent, $y=\lambda x+\mu z$ with
  \[
  \lambda(\lambda+1)\|x\|^2+\mu(\mu+1)\|z\|^2+2\lambda(\mu+1)\langle x,z\rangle=0
  \quad\text{and}\quad \mu\ge\lambda .
  \]
\item[(d)] no two of $x,y,z$ are linearly independent and
  $\langle x,y\rangle\langle z,w\rangle\ge0$.
\end{enumerate}
In particular, equality occurs only when $\det G=0$, i.e.\ in flat configurations.
\end{proposition}

\begin{proof}
Since $L^2-R^2=2(L_1-R_1)$ with $L,R,L_1\ge0$, equality holds iff $L_1=R_1$.
If $\det G\neq0$, then $\xi>0$ by the proof of Theorem~\ref{theorem-strong-hlawka}, i.e.
$L_1>|R_1|\ge R_1$, so equality fails. Hence equality forces linear dependence.
For dependent $x,y,z$, since $L_1\ge0$ and $L_1^2=R_1^2+\xi$,
\begin{equation}\label{eq:star}
L_1=R_1\iff \xi=0\ \text{and}\ R_1\ge0 .
\end{equation}

(a) If $x,y$ are independent, dependence gives $z=\lambda x+\mu y$, and
$a^2b^2-p^2>0$. By \eqref{eq:A2}, $\xi=0$ iff
$\|x\|^2\lambda(\lambda+1)=\|y\|^2\mu(\mu+1)$. Then \eqref{eq:star} gives (a). (By \eqref{eq:A1},
$R_1=(\lambda+\mu+2\lambda\mu)\|x\|^2\|y\|^2+\big(\lambda(\lambda+1)\|x\|^2+\mu(\mu+1)\|y\|^2\big)\langle x,y\rangle$.)

(b) If $y,z$ are independent, dependence gives $x=\lambda y+\mu z$, and $b^2c^2-r^2>0$.
By \eqref{eq:B2}, $\xi=0$ iff $E=0$, which is the first condition in (b). Then
\eqref{eq:B1} gives $R_1=(\mu-\lambda)\|y\|^2\|z\|^2$, and $\|y\|^2\|z\|^2>0$, so
$R_1\ge0$ iff $\mu\ge\lambda$. With \eqref{eq:star} this is (b).

(c) follows from (b) by exchanging $x$ and $y$, which leaves $L_1$ and $R_1$ unchanged.

(d) If no two of the vectors are independent, they lie on a line: $x=\alpha e$, $y=\beta e$,
$z=\gamma e$ with $\|e\|=1$. Then $\|x\|\|y\|=|\alpha\beta|$, $\|z\|\|w\|=|\gamma(\alpha+\beta+\gamma)|$ and
$\|x+z\|\|y+z\|=|\alpha\beta+\gamma(\alpha+\beta+\gamma)|$. By the equality case of the triangle
inequality in $\mathbb R$, equality holds iff $\alpha\beta\,\gamma(\alpha+\beta+\gamma)\ge0$, i.e.\ iff
$\langle x,y\rangle\langle z,w\rangle\ge0$.

Every dependent triple falls under (a), (b), (c) or (d) (test in this order: $x,y$ independent; else
$y,z$; else $x,z$; else none). Conversely, in each of (a)--(c) the stated conditions give $\xi=0$ and
$R_1\ge0$, hence equality by \eqref{eq:star}. If several of (a)--(d) apply, their conditions are
equivalent, since each is equivalent to equality.
\end{proof}

\section{An inequality based on Gram matrices}
\label{section-3}

\begin{theorem}
We have, for all vectors $x,y,z \in H$ and real numbers $\alpha,\beta,\gamma$:
\[ \alpha^2  \|x\|^2 \langle y, z \rangle^2 + \beta^2 \|y\|^2 \langle x, z \rangle^2 + \gamma^2 \|z\|^2 \langle x, y \rangle^2  + 2  (\alpha\beta + \alpha\gamma + \beta\gamma) \langle x, y\rangle \langle x, z\rangle \langle y, z\rangle \geq 0. \]
\end{theorem}

\begin{proof}
Given \(x, y, z\) in $H$, the Gram matrix associated to $(x,y,z)$ is  
\[
G = \begin{pmatrix}
\langle x, x \rangle & \langle x, y \rangle & \langle x, z \rangle \\
\langle y, x \rangle & \langle y, y \rangle & \langle y, z \rangle \\
\langle z, x \rangle & \langle z, y \rangle & \langle z, z \rangle
\end{pmatrix}.
\]
It is positive semidefinite, so for any real numbers \(\alpha, \beta, \gamma\), we have  
\[
Q(\alpha, \beta, \gamma) = (\alpha, \beta, \gamma) \, G \, \begin{pmatrix} \alpha \\ \beta \\ \gamma \end{pmatrix} \ge 0.
\]
Expanding, we have
\[
Q(\alpha, \beta, \gamma) =
\alpha^2 \|x\|^2 + \beta^2 \|y\|^2 + \gamma^2 \|z\|^2
+ 2\alpha\beta\langle x, y\rangle + 2\alpha\gamma\langle x, z\rangle + 2\beta\gamma\langle y, z\rangle \geq 0.
\]
Replace
\[
\alpha \mapsto \alpha \langle y, z \rangle,
\]
\[
\beta \mapsto \beta \langle x, z \rangle,
\]
\[
\gamma \mapsto \gamma \langle x, y \rangle.
\]
Then
\begin{align*}
R(\alpha,\beta,\gamma) &= Q(\alpha \langle y, z \rangle, \beta \langle x, z \rangle , \gamma  \langle x, y \rangle) \\
&=  \alpha^2  \|x\|^2 \langle y, z \rangle^2 + \beta^2 \|y\|^2 \langle x, z \rangle^2 + \gamma^2 \|z\|^2 \langle x, y \rangle^2  \\
&\quad + 2  (\alpha\beta + \alpha\gamma + \beta\gamma) \langle x, y\rangle \langle x, z\rangle \langle y, z\rangle \geq 0.
\end{align*}
\end{proof}

\begin{corollary} \mbox{}\\*
\begin{enumerate}
\item If $\langle x, y\rangle \langle x, z\rangle \langle y, z\rangle > 0$, we have
\[ \|x\|^2 \langle y, z \rangle^2 + \|y\|^2 \langle x, z \rangle^2 + \|z\|^2 \langle x, y \rangle^2  \geq 3 \langle x, y\rangle \langle x, z\rangle \langle y, z\rangle. \]
\item If $\langle x, y\rangle \langle x, z\rangle \langle y, z\rangle < 0$, we have
\[ \|x\|^2 \langle y, z \rangle^2 + \|y\|^2 \langle x, z \rangle^2 + \|z\|^2 \langle x, y \rangle^2  \geq - 6 \langle x, y\rangle \langle x, z\rangle \langle y, z\rangle.\]
\end{enumerate}
These inequalities are sharp.
\end{corollary}

\begin{proof} \mbox{}\\*
\begin{enumerate}
\item If 
\[ \langle x, y\rangle \langle x, z\rangle \langle y, z\rangle \geq 0 \]
Then $R(1,1,-2) > 0$ implies
\[ \|x\|^2 \langle y, z \rangle^2 + \|y\|^2 \langle x, z \rangle^2 + 4 \|z\|^2 \langle x, y \rangle^2  \geq 6 \langle x, y\rangle \langle x, z\rangle \langle y, z\rangle.\]
Summing up symmetrically and dividing by 6, we get
\[ \|x\|^2 \langle y, z \rangle^2 + \|y\|^2 \langle x, z \rangle^2 + \|z\|^2 \langle x, y \rangle^2  \geq 3 \langle x, y\rangle \langle x, z\rangle \langle y, z\rangle.\]
\item If
\[ \langle x, y\rangle \langle x, z\rangle \langle y, z\rangle < 0 \]
Then $R(1,1,1) \geq 0$ implies
\[ \|x\|^2 \langle y, z \rangle^2 + \|y\|^2 \langle x, z \rangle^2 + \|z\|^2 \langle x, y \rangle^2  \geq - 6 \langle x, y\rangle \langle x, z\rangle \langle y, z\rangle. \]
\end{enumerate}
Next, we discuss the sharpness of these inequalities. For $D>0$, take \( x = y = z \) nonzero. It yields LHS = RHS. For $D < 0$, take \(x = (1,0,0)\), \(y = \left(\frac12, \frac{\sqrt{3}}{2}, 0\right)\), \(z = \left(\frac12, -\frac{\sqrt{3}}{2}, 0\right)\). \\ 
We have
\[
\langle x, y \rangle = 1\cdot \frac12 + 0\cdot \frac{\sqrt{3}}{2} = \frac12
\]
\[
\langle x, z \rangle = 1\cdot \frac12 + 0\cdot \left(-\frac{\sqrt{3}}{2}\right) = \frac12
\]
\[
\langle y, z \rangle = \frac12\cdot \frac12 + \frac{\sqrt{3}}{2} \cdot \left(-\frac{\sqrt{3}}{2}\right) = \frac14 - \frac{3}{4} = -\frac12
\]
So 
\[ RHS = -6\left(\frac12\right)\left(\frac12\right)\left(-\frac12\right) = \frac68 = \frac34. \]
Moreover
\[
\|x\|^2 \langle y, z\rangle^2 = 1\cdot \frac14 = \frac14
\]
\[
\|y\|^2 \langle x, z\rangle^2 = 1\cdot \frac14 = \frac14
\]
\[
\|z\|^2 \langle x, y\rangle^2 = 1\cdot \frac14 = \frac14
\]
So 
\[ LHS = \frac14 + \frac14 + \frac14  = \frac 34. \]
Equality holds.
\end{proof}

\begin{remark} \mbox{}\\* 
\begin{enumerate}
\item This corollary can be understood as a reverse inequality of the Gram determinant inequality
\[ \det G \geq 0 \]
which translates to
\[ \|x\|^2 \|y\|^2 \|z\|^2 + 2\langle x, y \rangle \langle x, z \rangle \langle y, z \rangle \geq \|x\|^2 \langle y, z \rangle^2 + \|y\|^2 \langle x, z \rangle^2 + \|z\|^2 \langle x, y \rangle^2. 
\]
\item The first point of the corollary implies the Cauchy-Schwarz inequality, and therefore all other inequalities deduced from it. One can show this by taking $z=y$.
\end{enumerate}
\end{remark}

\nocite{*} 
\bibliographystyle{amsplain}
\bibliography{references}

\Addresses

\end{document}